\documentclass[preprint,12pt]{elsarticle}

\usepackage{amssymb}
\usepackage{amsthm}
\usepackage{amsmath,amsfonts,latexsym,amscd,color}

\usepackage{lineno}

\newtheorem{theorem}{Theorem}[section]

\newtheorem{corollary}[theorem]{Corollary}

\newtheorem{remark}[theorem]{Remark}

\theoremstyle{definition}
\newtheorem{definition}[theorem]{Definition}
\newtheorem{example}[theorem]{Example}
\newtheorem*{agreement*}{Agreement}

\newcommand{\norm}[1]{\lVert#1\rVert}
\newcommand{\norml}[1]{\left\lVert#1\right\rVert}
\newcommand{\abs}[1]{\left|#1\right|}

\begin{document}
\begin{frontmatter}



\title{Square roots of $H$-nonnegative matrices}


\author[label1,label3]{D.B.~Janse van Rensburg}
\author[label1,label3]{M.~van~Straaten}
\author[label1,label3]{F.~Theron}
\author[label2]{C.~Trunk}

\address[label1]{School~of~Mathematical~and~Statistical~Sciences,
North-West~University,
Research Focus: Pure and Applied Analytics,
Private~Bag~X6001,
Potchefstroom~2520,
South Africa.
E-mail: \texttt{dawie.jansevanrensburg@nwu.ac.za, madelein.vanstraaten@nwu.ac.za, frieda.theron@nwu.ac.za}}
\address[label2]{Institut f\"ur Mathematik, Technische Universit\"at Ilmenau, Postfach 10 05 65, D-98684 Ilmenau, Germany.  E-mail:
    \texttt{carsten.trunk@tu-ilmenau.de}}
\address[label3]{DSI-NRF Centre of Excellence in Mathematical and Statistical Sciences (CoE-MaSS)}

\begin{abstract}
Roots of matrices are well-studied. The conditions for their
existence are understood: The block sizes of nilpotent Jordan blocks, arranged in pairs,
have to satisfy some simple algebraic property.

More interesting are structured roots
of structured matrices. Probably the best known example is the existence and uniqueness of positive definite square roots of a positive definite matrix. If one drops the
requirement of positive definiteness of the square root, it turns out that there exists an abundance of square roots.
Here a description of all canonical forms of all
 square roots is possible and is straight forward.

$H$-nonnegative matrices are $H$-selfadjoint and are
nonnegative with respect to an indefinite inner product with
Gramian $H$. An $H$-nonnegative matrix $B$ allows a decomposition in
a negative definite, a nilpotent $H$-nonnegative, and a positive definite matrix,
$B=B_- \oplus B_0 \oplus B_+$. The interesting part is $B_0$, as only
Jordan blocks of size one and two occur. Determining a square root of $B$ reduces
to determining a square root of each of $B_-$, $B_0$, and $B_+$. Here we investigate for an $H$-nonnegative matrix: its square roots without additional structure, as well as its structured square roots that are $H$-nonnegative or $H$-selfadjoint.

For these three classes of square roots of $H$-nonnegative matrices we show a simple criterion
for their existence and describe all possible canonical forms.
This is based mainly on known results but an important new part is that in all three cases we
describe all possible square roots of the nilpotent $H$-nonnegative matrix $B_0$ explicitly.
Moreover, we show how our results can be applied to the conditional and
unconditional stability of $H$-nonnegative square roots of $H$-nonnegative matrices,
where the explicit description of the square roots of $B_0$ is used.

\end{abstract}




\end{frontmatter}


\section{Introduction}
The study of square roots of matrices dates back to 1858 \cite{Cayley}. Since 1858 numerous %
authors have studied square roots of matrices. For example, a necessary %
and sufficient condition for the existence of square roots of a complex matrix $A$ is given in \cite{CrossLan}. %
This condition relates to the dimensions of nullspaces of powers of $A$. Other authors %
such as \cite{BjorkHammar,Denman,High1,High2,NFB}, %
to mention but a few, focus on an efficient algorithm or a formula for square roots of a matrix.

The problem of finding roots of matrices where certain extra requirements or structure are %
imposed on the class of matrices studied and/or their roots, has also attracted a lot of attention. For example, \cite{AleSch} studies roots of $M$-matrices that are $M$-matrices %
themselves and \cite{FMMX} find results on Hamiltonian square roots of skew-Hamiltonian matrices.

Turning to indefinite inner product spaces, by definition  $[x,y]=(Hx,y)$ represents the indefinite inner product between vectors $x$ and $y$ in $\mathbb{C}^n$ where $H$ is some Hermitian and invertible matrix. Here $(x,y)$ is the usual Euclidean product in $\mathbb{C}^n$.  A complex ${n\times n}$ matrix $B$ is called \emph{$H$-selfadjoint} if $[Bx,y]=[x,By]$, that is, if $HB = B^*H$.
In \cite{GenTh} and \cite{MRR} $H$-selfadjoint square roots of $H$-selfadjoint matrices are studied. The canonical form of all $H$-selfadjoint matrices having $H$-selfadjoint square roots is derived, as well as the canonical form of the $H$-selfadjoint square roots.

A complex matrix $B$ is called \emph{$H$-nonnegative} if $[Bx,x]\geq 0$ for all $x\in\mathbb{C}^n$. The class of $H$-nonnegative matrices is a subclass of the class of $H$-selfadjoint matrices and has some very special properties, such as having a real spectrum and having Jordan chains of length at most two. We study square roots of $H$-nonnegative matrices and in Section~3 we find necessary and sufficient conditions for the existence of square roots. %
We also give explicit formulas for the square roots of nilpotent matrices. In Section~4 we characterize all $H$-selfadjoint square roots of $H$-nonnegative matrices and find characterizations of an %
$H$-nonnegative matrix with entries in $\mathbb{C}$ to have $H$-selfadjoint square roots, while in Section~5 we find conditions for an $H$-nonnegative matrix to have an $H$-nonnegative square root. Finally, Section~6 contains the theory regarding the stability of these $H$-nonnegative square roots.

\section{Preliminaries}
The Jordan form plays an important role in this paper.  We denote an $n\times n$ Jordan block corresponding to an eigenvalue $\lambda$ in a Jordan form by $J_n(\lambda)$ and the $n\times n$ sip matrix which has ones on the main anti-diagonal and zeros elsewhere by $Q_n$.

Let $B$ be a square matrix which at eigenvalue $\lambda$ has in its Jordan form Jordan blocks of sizes in decreasing order $(n_1,n_2,n_3,\ldots,n_r)$. We call this the \emph{Segre characteristic} of matrix $B$ corresponding to the eigenvalue $\lambda$ (see for example the book by Shapiro \cite{Shapiro}).

In the following theorem we recall the well-known canonical form given by \eqref{canonicalC} and \eqref{canonicalH} for a pair $(B,H)$, where $B$ is an $H$-selfadjoint matrix and $H$ is an invertible Hermitian matrix (see e.g.\ \cite{GenTh}).
\begin{theorem}\label{canonform}
Let $H$ be an invertible Hermitian $n\times n$ matrix, and let $B$ be an $n\times n$ complex $H$-selfadjoint matrix. Then there exists an invertible $n\times n$ matrix $S$ such that $S^{-1}BS$ and $S^*HS$ have the form
\begin{eqnarray}\label{canonicalC}
S^{-1}BS &=& J_{k_1}(\lambda_1) \oplus \cdots \oplus J_{k_{\alpha}}(\lambda_{\alpha})\nonumber \\
&\oplus& [J_{k_{\alpha+1}}(\lambda_{\alpha+1}) \oplus J_{k_{\alpha+1}}(\overline{\lambda}_{\alpha+1})] \oplus \cdots \oplus
[J_{k_{\beta}}(\lambda_{\beta}) \oplus J_{k_{\beta}}(\overline{\lambda}_{\beta})]
\end{eqnarray}
 where $\lambda_1,\hdots,\lambda_{\alpha}$ are real and $\lambda_{\alpha+1},\hdots,\lambda_{\beta}$ are nonreal with positive imaginary parts; and
\begin{equation}\label{canonicalH}
S^*HS = \varepsilon_1 Q_{k_1} \oplus \cdots \oplus \varepsilon_{\alpha}Q_{k_{\alpha}} \oplus Q_{2k_{\alpha+1}} \oplus \cdots \oplus Q_{2k_{\beta}}
\end{equation}
where $\varepsilon_1,\hdots,\varepsilon_{\alpha}$ are $\pm 1$. For a given pair $(B,H)$, the canonical form given by \eqref{canonicalC} and \eqref{canonicalH} is unique up to permutation of orthogonal components in \eqref{canonicalH} and the same simultaneous permutation of the corresponding blocks in \eqref{canonicalC}.
\end{theorem}

Observe that only the blocks in \eqref{canonicalH} related to blocks in \eqref{canonicalC} with a real eigenvalue possess a sign. The vector of signs $(\varepsilon_1,\ldots,\varepsilon_\alpha)$ is called the {\it sign characteristic} of $(B,H)$.

In the canonical form given by \eqref{canonicalC} and \eqref{canonicalH}, for any $i=1,\ldots,\alpha$, we say the block $J_{k_i}(\lambda_i)$ has \emph{positive sign characteristic} if $\varepsilon_i=1$,  and  \emph{negative sign characteristic} if $\varepsilon_i=-1$.

$H$-nonnegative matrices have the following characteristics (taken from \cite[Theorem 5.7.2]{GLR}).
\begin{theorem}\label{ThmHnonneg}
A matrix $B$ is $H$-nonnegative if and only if the following conditions hold:
\begin{enumerate}
\item[{\rm (i)}] $B$ is $H$-selfadjoint;
\item[{\rm (ii)}] $B$ has a real spectrum;
\item[{\rm (iii)}] the canonical form of the pair $(B,H)$ is
\begin{equation*}
S^{-1}BS=\bigoplus_{i=1}^qJ_1(\lambda_i)\oplus\bigoplus_{i=1}^{r-q}J_1(\lambda_{q+i})\oplus\bigoplus_{i=1}^sJ_1(0)\oplus\bigoplus_{i=1}^tJ_2(0),
\end{equation*}
\begin{equation*}
S^*HS=\bigoplus_{i=1}^qQ_1\oplus\bigoplus_{i=1}^{r-q}(-Q_1)\oplus\bigoplus_{i=1}^s\varepsilon_iQ_1\oplus\bigoplus_{i=1}^tQ_2,
\end{equation*}
where $\lambda_1,\ldots,\lambda_q>0,$ $\lambda_{q+1},\ldots,\lambda_r<0$, and $\varepsilon_i=\pm1$.
Thus $B$ has positive sign characteristic at a nonzero eigenvalue $\lambda$  if $\lambda>0$, and negative sign characteristic if $\lambda<0$; each Jordan block of size $2$ at the zero eigenvalue has positive sign characteristic.
\end{enumerate}
\end{theorem}

We frequently require Theorem~3.1 from \cite{MRR} in some of the sections and therefore present it here in an equivalent form without the ``moreover'' part.

\begin{theorem}\label{Thm3.1inMRR}
Let $B$ be an $H$-selfadjoint matrix. Then there exists an $H$-selfadjoint matrix $A$ such that $A^2=B$ if and only if the canonical form of  $(B,H)$ has the following properties:
\begin{enumerate}
\item[{\rm (i)}] the Jordan blocks corresponding to the negative eigenvalues exist in pairs and the two Jordan blocks in each pair have opposite sign characteristic;
\item[{\rm (ii)}] the Jordan blocks corresponding to the zero eigenvalue can be written as $J^{(1)}\oplus J^{(2)}\oplus J^{(3)}$, where $J^{(1)}$ is a direct sum of pairs $J_{p_i}(0)\oplus J_{p_i}(0)$, $J^{(2)}$ is a direct sum of pairs $J_{p_i}(0)\oplus J_{p_i-1}(0)$ and $J^{(3)}$ is a direct sum of $1\times 1$ blocks and where the blocks in each pair in $J^{(1)}$ have opposite sign characteristic and those in each pair in $J^{(2)}$  have the same sign characteristic.
\end{enumerate}
\end{theorem}

\section{Square roots}\label{SectionSqR}
In this section we describe all square roots of $H$-nonnegative matrices, i.e.\ the collection of all matrices $A$ such that $A^2=B$ and $B$ is $H$-nonnegative. Without loss of generality we always assume $B$ to be in Jordan form. The following result is well-known (see e.g.\ \cite{BorweinR}).

\begin{theorem}\label{ThmGenDecomp}
Let $B$ be $H$-nonnegative and in Jordan form. Write $B$ as
\begin{equation}\label{eqBdecomp}
B=B_-\oplus B_0\oplus B_+,
\end{equation}
where $\sigma(B_{\pm})\subset\mathbb{R}^{\pm}$ and $\sigma(B_0)=\{0\}$. A matrix $A$ is a square root of $B$ if and only if there exist square roots $A_-$, $A_0$, and $A_+$ of $B_-$, $B_0$, and $B_+$, respectively, such that
\begin{equation}\label{eqAdecomp}
A=A_-\oplus A_0\oplus A_+.
\end{equation}
\end{theorem}

\begin{proof}
The partition of $B$ follows from Theorem~\ref{ThmHnonneg}. Given square roots $A_-$, $A_0$, and $A_+$ of $B_-$, $B_0$, and $B_+$, respectively, it is obvious that $A_-\oplus A_0\oplus A_+$ is a square root of $B$. We show the converse. Let $A$ be a square root of $B$, i.e.\ $A^2=B$. Therefore, $A$ and $B$ commute and, by \cite[Theorem 3 in Chapter VIII]{Gantmacher}, $A$ allows a representation as in \eqref{eqAdecomp}. As $A^2=A^2_-\oplus A^2_0\oplus A^2_+$ we conclude that $A_-$, $A_0$, and $A_+$ are square roots of $B_-$, $B_0$, and $B_+$, respectively.
\end{proof}

It follows from Theorem~\ref{ThmHnonneg} that $B_-$ and $B_+$ are diagonalizable and, hence, they have a square root. Therefore a square root of $B$ exists if and only if a square root of $B_0$ exists according to Theorem~\ref{ThmGenDecomp}. For this reason, we focus on the case where $B=B_0$ is a nilpotent $H$-nonnegative matrix. Let $(n_1,n_2,\ldots,n_r)$ be the Segre characteristic of $B_0$. We call a collection $\mathcal{S}$ of pairs of numbers from $\mathcal{N}=\{n_i\}_{i=1}^r\cup \{0\}$ a \emph{Segre pairing} if:
\begin{itemize}
\item each entry in the Segre characteristic appears in one and only one pair of $\mathcal{S}$;
\item each pair in $\mathcal{S}$ consists of two entries from the Segre characteristic or one entry from the Segre characteristic and one zero, with the first entry greater or equal to the second entry;
\item the absolute difference between the two entries in each pair is at most one;
\item the pairs are arranged lexicographically.
\end{itemize}
Note that a pair of the form $(4,2)$, $(2,0)$ or $(0,0)$ is not a member of a Segre pairing. In \cite{BorweinR} it is shown that a matrix $B_0$ has a square root if and only if there exists a Segre pairing. By Theorem~\ref{ThmHnonneg} we know that for each pair $(m,n)$ in a Segre pairing of an $H$-nonnegative matrix $B_0$, we have $m\leq 2$, $n\leq 2$ and $\abs{m-n}\leq 1$. This leads to the following characterization of square roots of a nilpotent $H$-nonnegative matrix.

\begin{theorem}\label{thSqRConditions}
Let $B_0$ be a nilpotent $H$-nonnegative matrix. Then the following are equivalent:
\begin{enumerate}
\item[{\rm (i)}] $B_0$ has a square root;
\item[{\rm(ii)}] There exists a Segre pairing;
\item[{\rm(iii)}] In the Segre characteristic of $B_0$, either there is an even number of entries equal to two, or there is an odd number of entries equal to two and at least one entry equal to one.
\end{enumerate}
\end{theorem}

\begin{proof}
The items (i) and (ii) are equivalent, see \cite{BorweinR}.

By the definition of Segre pairing, assertion (ii) implies assertion (iii).

For the proof that (iii) implies (ii), we pair all the entries in the Segre characteristic equal to two. If there is an odd number of entries equal to two, pair one 2 in the Segre characteristic with a 1. Pair the rest of the entries equal to one with other entries equal to one or with a zero. This forms a Segre pairing.
\end{proof}

In the following theorem we describe all square roots of a nilpotent $H$-nonnegative matrix
related to a pair in $\mathcal{S}$, which can be made explicit. As a by-product we also
obtain the Jordan form.

\begin{theorem}\label{ThmDescrSquareRoot}
Let $B_0$ be a nilpotent $H$-nonnegative matrix in Jordan form which has a Segre pairing $\mathcal{S}$.
Any pair of $\mathcal{S}$ is of the form
\begin{equation}\label{Porque}
(2,2),\, (2,1),\, (1,1) \mbox{ or } (1,0).
\end{equation}
These pairs correspond to Jordan blocks of the form
$J_2(0)\oplus J_2(0),\, J_2(0)\oplus J_1(0),\, J_1(0)\oplus J_1(0)$ and $J_1(0)$, respectively, in the matrix $B_0$. The square roots of each possible pair in $\mathcal{S}$ are as follows:
\begin{enumerate}
\item[\rm (i)] The square root of $J_1(0)=[0]$, which is associated with the pair $(1,0)$, is equal to $[0]$. Its Jordan form is $J_1(0)$ again.
\item[\rm (ii)] The square root of $J_1(0)\oplus J_1(0)$, which is associated with the pair $(1,1)$, is
\begin{equation}\label{eq11genform}
\begin{bmatrix}
\alpha & \frac{-\alpha^2}{\beta} \\ \beta & -\alpha
\end{bmatrix}, \textit{ for any complex numbers }\alpha,\beta\textit{ with }\beta\neq0,
\end{equation}
\begin{equation}\label{eq11genform2}
\textit{or } \begin{bmatrix}
0 & \alpha\\ 0 & 0
\end{bmatrix}, \textit{ for any complex number }\alpha.
\end{equation}
The Jordan form of all matrices in \eqref{eq11genform} and  \eqref{eq11genform2} is $J_2(0)$, except in the case in \eqref{eq11genform2} where $\alpha=0$. In this case, the Jordan form is $J_1(0)\oplus J_1(0)$.
\item[\rm (iii)] The square root of $J_2(0)\oplus J_1(0)$, which is associated with the pair $(2,1)$, is
\begin{equation}\label{eq21genform}
\begin{bmatrix}
0&\alpha & \beta \\ 0&0&0 \\ 0&\frac{1}{\beta}&0
\end{bmatrix},\textit{ for complex numbers }\alpha\textit{ and }\beta\neq0.
\end{equation}
The Jordan form of all matrices in \eqref{eq21genform} is $J_3(0)$.
\item[\rm (iv)] The square root of $J_2(0)\oplus J_2(0)$, which is associated with the pair $(2,2)$, is
\begin{equation}\label{eq22genform}
\begin{bmatrix}
-\alpha_3&-\alpha_4&\frac{-\alpha_3^2}{\alpha_1}&\beta\\
0&-\alpha_3&0&\frac{-\alpha_3^2}{\alpha_1} \\
\alpha_1 & \alpha_2 & \alpha_3&\alpha_4 \\
0&\alpha_1&0&\alpha_3
\end{bmatrix},
\end{equation}
for a nonzero $\alpha_1$ and where $\beta=\frac{1}{\alpha_1^2}(\alpha_1+\alpha_3^2\alpha_2-2\alpha_1\alpha_3\alpha_4)$, or
\begin{equation}\label{eq22genform2}
\begin{bmatrix}
0 & \gamma_1 & \gamma_2 & \gamma_3 \\
0 & 0 & 0 & \gamma_2 \\
0 & \frac{1}{\gamma_2} & 0 & -\gamma_1 \\
0 & 0 & 0 & 0
\end{bmatrix},
\end{equation}
for complex numbers $\gamma_i$, and nonzero $\gamma_1$,$\gamma_2$. The Jordan form of all matrices in \eqref{eq22genform} and \eqref{eq22genform2} is $J_4(0)$.
\end{enumerate}
\end{theorem}

\begin{proof}
The first statement follows from the form of $B_0$ according to
Theorem~\ref{ThmHnonneg}.

Part (i) is trivial since the square root of a $1\times 1$ zero matrix is the $1\times 1$ zero matrix.

Part (ii). Consider a nilpotent matrix in Jordan form with Segre characteristic $(1,1)$, i.e.\ the $2\times 2$ zero matrix. Let $A$ be a square root, that is, $A^2=\left[\begin{smallmatrix}
0 & 0 \\ 0 & 0
\end{smallmatrix}\right]$. Then for the two unit vectors $e_1,e_2$ in $\mathbb{C}^2$ we have
$A^2e_1=0$ and $A^2e_2=0$. Let
\begin{equation}\label{eq2Aej}
\begin{split}
Ae_1=\alpha_1 e_1+\beta_1e_2,\\
Ae_2=\alpha_2e_1+\beta_2e_2,
\end{split}
\end{equation}
where the coefficients $\alpha_j$ and $\beta_j$ are complex numbers. Multiply these equations by $A$ and use them again. Then we obtain the equations
\begin{equation*}
\begin{split}
\alpha_1^2+\alpha_2\beta_1=0,\\
\beta_2^2+\alpha_2\beta_1=0,\\
\beta_1(\alpha_1+\beta_2)=0,\\
\alpha_2(\alpha_1+\beta_2)=0.
\end{split}
\end{equation*}
Assume that $\beta_1$ is nonzero, then $\alpha_1=-\beta_2$ and $\alpha_2=-\beta_2^2/\beta_1$. This gives the form of any square root of the zero matrix of order $2$:
\begin{equation*}
\begin{bmatrix}
-\beta_2 & \frac{-\beta_2^2}{\beta_1}\\
\beta_1 & \beta_2
\end{bmatrix},
\end{equation*}
which shows \eqref{eq11genform}. Moreover, \eqref{eq2Aej} implies $Ae_1\neq 0$. As $A^2e_1=0$ we see that the Jordan form of the matrix in \eqref{eq11genform} is $J_2(0)$.
Note that if $\beta_1=0$, then $\alpha_1=\beta_2=0$, and the square root is
\begin{equation*}
\begin{bmatrix}
0 & \alpha_2 \\
0 & 0
\end{bmatrix},
\end{equation*}
for any complex number $\alpha_2$.

Part (iii). Consider a nilpotent matrix in Jordan form with Segre characteristic $(2,1)$: $$J_2(0)\oplus J_1(0)=\begin{bmatrix}
0 & 1 & 0 \\ 0 & 0 & 0\\ 0 & 0 & 0
\end{bmatrix},$$
and let $A$ be a square root such that $A^2=J_2(0)\oplus J_1(0)$. Let $e_1,e_2,e_3$ denote the three unit vectors in $\mathbb{C}^3$. We have \begin{equation}\label{eq3Jchains}
A^2e_1=0;\quad A^2e_2=e_1;\quad  A^2e_3=0.
\end{equation}
Let
\begin{equation}\label{eq3Aej}
Ae_j=\alpha_je_1+\beta_je_2+\gamma_je_3,
\end{equation}
for $j=1,2,3$, where $\alpha_j,\beta_j,\gamma_j$ are complex numbers. Multiply every equation in \eqref{eq3Jchains} by the matrix $A$ and compare with the equations in \eqref{eq3Aej} multiplied by $A^2$, then
\begin{equation}\label{eq3betas}
\beta_1e_1=0;\quad\beta_3e_1=0;\quad\beta_2e_1=Ae_1.
\end{equation}
From the first two equations it follows that $\beta_1=0$, $\beta_3=0$, and from the last equation in \eqref{eq3betas} using \eqref{eq3Aej} with $j=1$ we get
\begin{equation*}
\beta_2e_1=\alpha_1e_1+\gamma_1e_3,
\end{equation*}
which gives $\gamma_1=0$ and $\beta_2=\alpha_1$. Also, we multiply the last equation in \eqref{eq3betas} by $A$ and use \eqref{eq3Jchains} to get $\beta_2=0$ and consequently $\alpha_1=0$. Hence $Ae_1=0$. Next, we take the expressions for $Ae_2$ and $Ae_3$ in \eqref{eq3Aej}, multiply both sides by $A$, and use the equations in \eqref{eq3Aej}, \eqref{eq3Jchains} and $Ae_1=0$ to obtain
\begin{equation*}
\begin{split}
0=A^2e_3=\gamma_3Ae_3=\gamma_3\alpha_3e_1+\gamma_3^2e_3,\\
e_1=A^2e_2=\gamma_2Ae_3=\gamma_2\alpha_3e_1+\gamma_2\gamma_3e_3.
\end{split}
\end{equation*}
From the first equation we conclude that $\gamma_3=0$ and then it follows from the second equation that $\gamma_2\alpha_3=1$. Thus any square root of the matrix $J_2(0)\oplus J_1(0)$ is given by:
\begin{equation}\label{eq321genform}
\begin{bmatrix}
0 & \alpha_2 & \alpha_3 \\
0 & 0 & 0 \\
0 & \gamma_2 & 0
\end{bmatrix},
\end{equation}
where $\alpha_2,\alpha_3,\gamma_2$ are any complex numbers such that $\gamma_2\alpha_3=1$. Then
\begin{equation*}
Ae_2=\alpha_2 e_1+ \gamma_2e_3,\quad A(\alpha_2 e_1+ \gamma_2e_3)=\gamma_2 \alpha_3 e_1=e_1,\quad Ae_1=0,
\end{equation*}
and from $\gamma_2\alpha_3=1$, it follows that \eqref{eq321genform} has Jordan form $J_3(0)$.

Part (iv). Consider a nilpotent matrix in Jordan form with Segre characteristic $(2,2)$:
\begin{equation*}
J_2(0)\oplus J_2(0)=\begin{bmatrix}
0& 1 & 0 & 0 \\
0 & 0 & 0 & 0 \\
0 & 0 & 0 & 1 \\
0 & 0 & 0 & 0
\end{bmatrix},
\end{equation*}
and let $A$ be a square root such that $A^2=J_2(0)\oplus J_2(0)$. Let $e_1,\ldots,e_4$ be the four unit vectors in $\mathbb{C}^4$. Then we have
\begin{equation}\label{eq4Jchains}
A^2e_1=0;\quad A^2e_2=e_1;\quad A^2e_3=0;\quad A^2e_4=e_3.
\end{equation}
Now, let
\begin{equation}\label{eq4Aej}
Ae_j=\alpha_je_1+\beta_je_2+\gamma_je_3+\delta_je_4,
\end{equation}
for $j=1,2,3,4$, and where all coefficients are complex numbers. Now we multiply each equation in \eqref{eq4Aej} by $A^2$ and using \eqref{eq4Jchains} we obtain
\begin{equation}\label{eq4A^3s}
A^3e_j=\beta_je_1+\delta_je_3,\quad j=1,\ldots,4.
\end{equation}
 By taking $j=1$ and $j=3$ in \eqref{eq4A^3s} and using \eqref{eq4Jchains} we obtain $\beta_1=\delta_1=\beta_3=\delta_3=0$. Using \eqref{eq4Jchains} again, the equations in \eqref{eq4A^3s} for $j=2$ and $j=4$ become
\begin{equation}\label{eq4right}
Ae_1=\beta_2e_1+\delta_2e_3\quad\textup{and} \quad Ae_3=\beta_4e_1+\delta_4 e_3.
\end{equation}
From \eqref{eq4Aej} we also have
 \begin{equation}\label{eq4left}
 Ae_1=\alpha_1e_1+\gamma_1e_3,\quad Ae_3=\alpha_3e_1+\gamma_3e_3,
\end{equation}
and when we compare the coefficients of the unit vectors in \eqref{eq4right} and \eqref{eq4left} we obtain the following
\begin{equation*}
\beta_2=\alpha_1,\;\delta_2=\gamma_1,\;\beta_4=\alpha_3,\;\delta_4=\gamma_3.
\end{equation*}
Then the equations in \eqref{eq4Aej} with $j=2$ and $j=4$ become
\begin{equation}\label{eq4newAej}
\begin{split}
Ae_2=\alpha_2e_1+\alpha_1e_2+\gamma_2e_3+\gamma_1e_4,\\
Ae_4=\alpha_4e_1+\alpha_3e_2+\gamma_4e_3+\gamma_3e_4.
\end{split}
\end{equation}
Finally, we multiply the expressions for $Ae_1$ and $Ae_3$ by $A$ and assume that $\gamma_1\neq0$. Use the equations in \eqref{eq4Jchains} and  then
\begin{eqnarray*}
j=1: & 0=\alpha_1Ae_1+\gamma_1Ae_3=(\alpha_1^2+\gamma_1\alpha_3)e_1+(\alpha_1\gamma_1+\gamma_1\gamma_3)e_3 ,\\
j=3: & 0=\alpha_3Ae_1+\gamma_3Ae_3=(\alpha_3\alpha_1+\gamma_3\alpha_3)e_1+(\alpha_3\gamma_1+\gamma_3^2)e_3.
\end{eqnarray*}
From this we obtain equalities which gives
\begin{equation*}
\alpha_1=-\gamma_3,\quad \alpha_3=\frac{-\gamma_3^2}{\gamma_1}.
\end{equation*}
Using these expressions together with \eqref{eq4Jchains}, and doing the same with $Ae_2$ and $Ae_4$ from \eqref{eq4newAej}, we have
\begin{equation*}
\alpha_2=-\gamma_4,\quad\textup{and} \quad\alpha_4=\frac{1}{\gamma_1^2}\left(\gamma_1+\gamma_3^2\gamma_2-2\gamma_4\gamma_3\gamma_1\right).
\end{equation*}
By renaming the variables, we obtain the matrix in \eqref{eq22genform}. Now, note that the following hold
\begin{equation*}
\begin{split}
Ae_2=\alpha_2e_1-\gamma_3e_2+\gamma_2e_3+\gamma_1e_4,\\
A(\alpha_2e_1-\gamma_3e_2+\gamma_2e_3+\gamma_1e_4)=A^2e_2=e_1,\\
Ae_1=\alpha_1 e_1+\gamma_1 e_3,\\
A(\alpha_1 e_1+\gamma_1 e_3)=A^2 e_1=0,
\end{split}
\end{equation*}
and from $\gamma_1\neq 0$ it follows that \eqref{eq22genform} has Jordan form $J_4(0)$.

If $\gamma_1=0$ and $\alpha_3\neq0$, then we obtain $\gamma_2=\frac{1}{\alpha_3}$ and the square root is
\begin{equation*}
\begin{bmatrix}
0 & \alpha_2 & \alpha_3 & \alpha_4 \\
0 & 0 & 0 & \alpha_3 \\
0 & \frac{1}{\alpha_3} & 0 & -\alpha_2 \\
0 & 0 & 0 & 0
\end{bmatrix}. \end{equation*}
Note that $\gamma_3=0$, hence the following hold
\begin{equation*}
\begin{split}
Ae_4=\alpha_4e_1+\alpha_3e_2+\gamma_4e_3,\\
A(\alpha_4e_1+\alpha_3e_2+\gamma_4e_3)=A^2e_4=e_3,\\
Ae_3=\alpha_3 e_1,\\
A(\alpha_3 e_1)=0,
\end{split}
\end{equation*}
and from $\alpha_3\neq 0$ it follows that \eqref{eq22genform2} has Jordan form $J_4(0)$.

If $\gamma_1=\alpha_3=0$, then it follows that $Ae_1=Ae_3=0$ which contradicts the equations in \eqref{eq4newAej} multiplied by $A$.
\end{proof}

We remark that the canonical forms found for the square roots in all cases correspond with those given for the $H$-selfadjoint square roots found in \cite[Theorem~3.1]{MRR}.

\begin{agreement*} Note that the Jordan forms of square roots of all types of pairs in any Segre pairing are uniquely determined, except in the case for $(1,1)$ (the second item of Theorem~\ref{ThmDescrSquareRoot}), where it can be $J_2(0)$ or $J_1(0)\oplus J_1(0)$. But a pair $(1,1)$ in any Segre pairing could just as well be replaced by two pairs $(1,0),(1,0)$, which will also have a square root with Jordan form $J_1(0)\oplus J_1(0)$ (the first item of Theorem~\ref{ThmDescrSquareRoot}). Therefore we agree not to allow $\left[\begin{smallmatrix}
0&0\\0&0
\end{smallmatrix}\right]$  as a square root for a pair $(1,1)$, i.e.\ we exclude $\alpha=0$ in \eqref{eq11genform2} of Theorem~\ref{ThmDescrSquareRoot}. Thus,
under this agreement, we have that the Jordan forms corresponding to all pairs in a Segre pairing are uniquely determined.
\end{agreement*}

We say that a square root is \emph{associated with a Segre pairing} if it is a direct sum of square roots associated with each pair of the Segre pairing (see Theorem~\ref{ThmDescrSquareRoot}). Then all square roots of nilpotent matrices associated with a Segre pairing have the same Jordan form.

We then obtain the following result from Theorem \ref{ThmDescrSquareRoot}:

\begin{corollary}\label{CorOnetoOne}
Let $B_0$ be a nilpotent $H$-nonnegative matrix. Then there is a one-to-one correspondence between the Jordan forms of the square roots $A_0$ of $B_0$ and the Segre pairings of $B_0$.
\end{corollary}

Thus two different Segre pairings of a matrix $B_0$ give two different Jordan forms of square roots of $B_0$.

\begin{example}
Let $B_0=J_2(0)\oplus J_2(0)\oplus J_1(0)\oplus J_1(0)$ and $H=Q_2\oplus Q_2\oplus \varepsilon_1Q_1\oplus \varepsilon_2Q_1$, where $\varepsilon_1,\varepsilon_2=\pm 1$. Then $B_0$ is $H$-nonnegative,
see Theorem~\ref{ThmHnonneg}. The Segre characteristic of $B_0$ is $(2,2,1,1)$. There exist three different Segre pairings which deliver a square root of $B_0$. Any square root of $B_0$ which is associated with the Segre pairing $(2,2),(1,1)$, will have Jordan form $J_4(0)\oplus J_2(0)$. Any square root which is associated with the Segre pairing $(2,2),(1,0),(1,0)$, will have Jordan form $J_4(0)\oplus J_1(0)\oplus J_1(0)$, and finally any square root which is associated with the Segre pairing $(2,1),(2,1)$,  will have Jordan form $J_3(0)\oplus J_3(0)$.
\end{example}

Combining the singular and nonsingular cases, we can formulate the following result which presents the Jordan form of all square roots. Let $\mathcal{S}$ be any Segre pairing of an $H$-nonnegative matrix $B$. Furthermore, let $\ell_1$ denote the number of pairs $(1,0)$ in $\mathcal{S}$, $\ell_2$ the number of $(1,1)$ pairs in $\mathcal{S}$, $\ell_3$ the number of $(2,1)$ pairs in $\mathcal{S}$, and $\ell_4$ the number of $(2,2)$ pairs in $\mathcal{S}$.

\begin{theorem}\label{ThmJordanFSqR}
Any square root of the $H$-nonnegative matrix $B$ has the Jordan form $A=A_{-}\oplus A_0
\oplus A_+$ with
\begin{equation*}
A_+= \bigoplus_{j=1}^{q}J_1(\delta_j\sqrt{\lambda_j}), \quad A_-=\bigoplus_{j=1}^{r-q}
J_1(i\delta_{q+j}\sqrt{-\lambda_{q+j}})
\end{equation*}
where $\lambda_1,\ldots,\lambda_q>0$ and $\lambda_{q+1},\ldots,\lambda_r<0$, counting multiplicities,
$\delta_j = \pm 1$, and $A_0$ is given by
\begin{equation*}
A_0=\bigoplus_{i=1}^{\ell_1}J_1(0)\oplus\bigoplus_{i=1}^{\ell_2} J_2(0)\oplus\bigoplus_{i=1}^{\ell_3} J_3(0)\oplus\bigoplus_{i=1}^{\ell_4} J_4(0).
\end{equation*}
\end{theorem}

The proof is trivial from Theorem~\ref{ThmDescrSquareRoot} and Corollary~\ref{CorOnetoOne}.

Note that $B_-$ and $B_+$ have a lot of square roots as can be seen in the following example.

\begin{example}
Let the following matrix be given
$$B=J_1(-3)\oplus J_1(-3)=\begin{bmatrix}
-3&\phantom{-}0\\\phantom{-}0&-3
\end{bmatrix}.$$
Then the following matrices are some of the square roots of $B$:
\begin{equation*}
\begin{bmatrix}
i\sqrt{3}&0\\0&i\sqrt{3}
\end{bmatrix},\; \begin{bmatrix}
-i\sqrt{3}&0\\0&-i\sqrt{3}
\end{bmatrix},\; \begin{bmatrix}
i\sqrt{3}&0\\0&-i\sqrt{3}
\end{bmatrix},\; \begin{bmatrix}
2i&\phantom{-}1\\1&-2i
\end{bmatrix},\; \begin{bmatrix}
2&-7\\1&-2
\end{bmatrix}.
\end{equation*}
\end{example}

\section{$H$-selfadjoint square roots}
Taking Section~\ref{SectionSqR} further, we now investigate $H$-selfadjoint square roots of $H$-nonnegative matrices $B$. We assume that the pair $(B,H)$ is in canonical form. Hence by Theorem~\ref{ThmHnonneg}
\begin{equation*}\label{eq4Bcanon}
B=\bigoplus_{i=1}^qJ_1(\lambda_i)\oplus\bigoplus_{i=1}^{r-q}J_1(\lambda_{q+i})\oplus\bigoplus_{i=1}^sJ_1(0)\oplus\bigoplus_{i=1}^tJ_2(0),
\end{equation*}
\begin{equation*}\label{eq4Hcanon}
H=\bigoplus_{i=1}^qQ_1\oplus\bigoplus_{i=1}^{r-q}(-Q_1)\oplus\bigoplus_{i=1}^s\varepsilon_iQ_1\oplus\bigoplus_{i=1}^tQ_2,
\end{equation*}
where $\lambda_1,\ldots,\lambda_q>0,$ $\lambda_{q+1},\ldots,\lambda_r<0$, and $\varepsilon_i=\pm1$.
For the following we denote by $s_+$  ($s_-$) the number of positive (resp.\ negative) entries in
the vector $(\varepsilon_1,\ldots,\varepsilon_s)$, hence we have
$$
s=s_-+s_+.
$$

\begin{theorem}\label{ThmHsaConditions}
Let $B$ be an $H$-nonnegative matrix in $\mathbb C^{n\times n}$. Then $B$ has an $H$-selfadjoint square root if and only if $B$ satisfies the following:
\begin{enumerate}
\item[{\rm(i)}] $B$ has no negative eigenvalues, that is, $\sigma(B)\subset [0,\infty)$;
\item[{\rm(ii)}] At the eigenvalue zero of $B$, the number of Jordan blocks of size one with positive sign characteristic is not less than the number of Jordan blocks of size two, i.e.,
$$s_+ \geq t.$$
\end{enumerate}
\end{theorem}

\begin{proof}
Let $B$ be an $H$-nonnegative matrix. Note that $B$ is also $H$-selfadjoint by Theorem~\ref{ThmHnonneg}. If conditions (i) and (ii) hold, then it follows from Theorem~\ref{Thm3.1inMRR} that $B$ has an $H$-selfadjoint square root.

Conversely, assume that the $H$-nonnegative matrix $B$ has an $H$-selfadjoint square root $A$. Because of the structure of $B$ as given in Theorem~\ref{ThmHnonneg}, it follows directly from Theorem~\ref{Thm3.1inMRR} that $B$ cannot have negative eigenvalues. Regarding the second assertion, Theorem~\ref{Thm3.1inMRR} also implies that the blocks corresponding to the pairs $(2,1)$ in the Segre pairing have the same sign characteristic. In fact, from Theorem~\ref{ThmHnonneg}, they all have positive sign characteristic. Consider a pair $(2,2)$ in a Segre pairing of $B$ which comes from $J_2(0)\oplus J_2(0)$. By Theorem~\ref{ThmHnonneg} each block has positive sign characteristic, but Theorem~\ref{Thm3.1inMRR} states that for $J_2(0)\oplus J_2(0)$ to have an $H$-selfadjoint square root, the two blocks must have opposite sign characteristic. This means that a pair $(2,2)$ can never be a member of any Segre pairing for an $H$-nonnegative matrix. Consequently, all pairs in a Segre pairing containing a $2$ are of the form $(2,1)$, so there must be in $B$ at least as many Jordan blocks of size one with positive sign characteristic as the number of Jordan blocks of size two.
\end{proof}

\begin{remark}
We mention that the ``only if" part of Theorem~\ref{ThmHsaConditions} can also be obtained by using known results on $H$-plus matrices and $H$-polar decompositions. Let $B$ be an $H$-nonnegative matrix. Assume that $B$ has an $H$-selfadjoint square root $A$. Then we can write $B=A^{[*]}A$. Since $[Bx,x]\geq 0$ for all $x\in\mathbb{C}^n$ we have from \cite{BMRRR2} that $A$ is an $H$-plus matrix since $\mu(A)=\inf_{[u,u]=1}[A^{[*]}Au,u]\geq 0$. The matrix $A$ admits a trivial $H$-polar decomposition: $A=I\cdot A$. Finally, by part (c) of Theorem~3.4 in \cite{BMRRR2}, see errata in \cite{BMRRR2err},  it follows that $B$ has no negative eigenvalues and by part (a) in the same theorem, it follows that $s_+\geq t$.
\end{remark}

 We now obtain an explicit description of all $H$-selfadjoint square roots of a nilpotent $H$-nonnegative matrix $B_0$.

\begin{theorem}\label{ThmDescrHsa}
Let $B_0$ be a nilpotent $H$-nonnegative matrix in Jordan form satisfying the second condition in Theorem~\ref{ThmHsaConditions} and having a Segre pairing $\mathcal{S}$.
Any pair of $\mathcal{S}$ is of the form
\begin{equation}\label{PorqueII}
(2,1),\, (1,1) \mbox{ or } (1,0).
\end{equation}
This corresponds to Jordan blocks of the form
$J_2(0)\oplus J_1(0),\, J_1(0)\oplus J_1(0)$ and $J_1(0)$, respectively, in the matrix $B_0$.
Then the $H$-selfadjoint square roots associated with each possible pair in $\mathcal{S}$ are as follows:
\begin{enumerate}
\item[\rm (i)] The $H$-selfadjoint square root of $J_1(0)=[0]$, which is associated with the pair $(1,0)$, is equal to $[0]$. Its Jordan form is $J_1(0)$ again.
\item[\rm (ii)] The $H$-selfadjoint square root of $J_1(0)\oplus J_1(0)$, which is associated with the pair $(1,1)$, and where the two blocks have opposite sign characteristic, is
\begin{equation*}
\begin{bmatrix}
\alpha & -\bar{\beta} \\ \beta & -\alpha
\end{bmatrix}, \textit{ for any real }\alpha\textit{ and complex }\beta,\;\abs{\beta}=\pm\alpha.
\end{equation*}
The Jordan form of all matrices of this form is $J_2(0)$.
\item[\rm (iii)] The $H$-selfadjoint square root of $J_2(0)\oplus J_1(0)$, which is associated with the pair $(2,1)$, and where the two blocks have both positive sign characteristic, is
\begin{equation*}
\begin{bmatrix}
0&\alpha & \beta \\ 0&0&0 \\ 0&\frac{1}{\beta}&0
\end{bmatrix},\textit{ for real }\alpha\textit{ and complex }\beta,\;\abs{\beta}=1.
\end{equation*}
The Jordan form of all matrices of this form is $J_3(0)$.
\end{enumerate}
\end{theorem}

\begin{proof}
Relation \eqref{PorqueII} follows from  \eqref{Porque} and the fact that
a pair $(2,2)$ can never be a member of any Segre pairing for an $H$-nonnegative matrix,
with an $H$-selfadjoint square root,
which was shown in the proof of Theorem~\ref{ThmHsaConditions}.

Assertion (i) is trivial. For (ii) we note that equations \eqref{eq11genform} and \eqref{eq11genform2} give the possible forms of square roots in this case. Firstly, for the matrix $A$ as in \eqref{eq11genform}, since $J_1(0)\oplus J_1(0)$ is $H$-selfadjoint where $H=\varepsilon_1 Q_1\oplus\varepsilon_2 Q_1$, for some $\varepsilon_i=\pm1$, we see that
\begin{equation*}
HA=\begin{bmatrix}
\varepsilon_1&0\\0 &\varepsilon_2
\end{bmatrix}\begin{bmatrix}
\alpha & \frac{-\alpha^2}{\beta} \\ \beta & -\alpha
\end{bmatrix}=\begin{bmatrix}
\varepsilon_1\alpha & \varepsilon_1\frac{-\alpha^2}{\beta} \\ \varepsilon_2\beta & -\varepsilon_2\alpha
\end{bmatrix}
\end{equation*}
is equal to $A^*H=(HA)^*$ if and only if $\alpha\in\mathbb{R}$, $\varepsilon_1\neq\varepsilon_2$ and $\abs{\beta}^2=\alpha^2$. This gives the $H$-selfadjoint square root associated with a pair $(1,1)$, as
\begin{equation*}\label{eq11Hsagenform}
\begin{bmatrix}
\alpha & -\bar{\beta} \\ \beta & -\alpha
\end{bmatrix},\quad \textup{ where } \alpha\in\mathbb{R},\;\abs{\beta}=\pm\alpha.
\end{equation*}
As for $A$ as in \eqref{eq11genform2}, the matrix $\left[\begin{smallmatrix}
0&\alpha\\0&0
\end{smallmatrix}\right]$ is $H$-selfadjoint if and only if $\alpha=0$, but we agreed to exclude this case for a pair $(1,1)$, since it is covered by two pairs $(1,0),(1,0)$.

For (iii) we note that $J_2(0)\oplus J_1(0)$ is $H$-selfadjoint where $H=Q_2\oplus \varepsilon Q_1$, but from the proof of Theorem~\ref{ThmHsaConditions} we know that $\varepsilon=1$. Let $A$ be the matrix in \eqref{eq21genform}, which is associated with a pair $(2,1)$. Then we have
\begin{equation*}
HA=\begin{bmatrix}
0 & 1 &0\\1&0&0\\0&0&1
\end{bmatrix}
\begin{bmatrix}
0&\alpha & \beta \\ 0&0&0 \\ 0&\frac{1}{\beta}&0
\end{bmatrix}=\begin{bmatrix}
0&0 & 0\\ 0&\alpha &\beta \\ 0&\frac{1}{\beta}&0
\end{bmatrix}.
\end{equation*}
This matrix is equal to $A^*H=(HA)^*$ if and only if $\alpha\in\mathbb{R}$ and $\abs{\beta}^2=1$. Therefore the $H$-selfadjoint square root associated with a pair $(2,1)$ is
\begin{equation*}\label{eq21Hsagenform}
\begin{bmatrix}
0&\alpha & \beta \\ 0&0&0 \\ 0&\frac{1}{\beta}&0
\end{bmatrix},\quad \alpha\textup{ real},\;\abs{\beta}=1.\qedhere
\end{equation*}
\end{proof}

The fact in (ii) of Theorem~\ref{ThmDescrHsa}, that the two blocks in $J_1(0)\oplus J_1(0)$ should have opposite sign characteristic for the existence of an $H$-selfadjoint square root, also follows from Theorem~\ref{Thm3.1inMRR}. Of course, when the blocks $J_1(0)$ in $B$ do not have the correct sign characteristic to be paired as $(1,1)$, they can be paired as two pairs $(1,0),(1,0)$.

For the following result where the canonical form of the $H$-selfadjoint square roots are given, let $\mathcal{S}$ be any Segre pairing of an $H$-nonnegative matrix $B$. Also, let $\ell_1$ denote the number of pairs $(1,0)$ in $\mathcal{S}$, $\ell_2$ the number of $(1,1)$ pairs in $\mathcal{S}$, and $\ell_3$ the number of $(2,1)$ pairs in $\mathcal{S}$.

\begin{theorem}\label{EternalLoveAngelicaZapata}
Let $B$ be an $H$-nonnegative matrix satisfying the conditions in Theorem~\ref{ThmHsaConditions}. Assume that $(B,H)$ is in canonical form. Then the canonical form of the pair $(A,H)$, where $A$ is an $H$-selfadjoint square root of $B$, is given by
\begin{equation*}
S^{-1}AS=\bigoplus_{i=1}^{q_1}J_1(\sqrt{\lambda_i})\oplus\bigoplus_{i=1}^{q_2}J_1(-\sqrt{\lambda_i})
\oplus\bigoplus_{i=1}^{\ell_1}J_1(0)\oplus\bigoplus_{i=1}^{\ell_2}J_2(0)\oplus\bigoplus_{i=1}^{\ell_3}J_3(0),
\end{equation*}
\begin{equation*}
S^*HS=\bigoplus_{i=1}^{q_1}Q_1\oplus \bigoplus_{i=1}^{q_2}Q_1\oplus\bigoplus_{i=1}^{\ell_1}\zeta_{i} Q_1\oplus\bigoplus_{i=1}^{\ell_2}\eta_iQ_2\oplus\bigoplus_{i=1}^{\ell_3}Q_3,
\end{equation*}
for some invertible matrix $S$, where $q_1+q_2=q$,  $\lambda_1,\ldots,\lambda_q>0$, the number $\zeta_i$ is equal to the sign characteristic of the corresponding block $J_1(0)$ in $B$, and $\eta_i=\pm1$.
\end{theorem}

\begin{proof}
As $\sigma(B)\subset[0,\infty)$, see Theorem~\ref{ThmHsaConditions}, the spectrum of $A$ is real.
This and Theorem~\ref{ThmDescrHsa} show the form of $S^{-1}AS$. The sign characteristics
related to the blocks $J_1(\sqrt{\lambda_i})$ and $J_1(-\sqrt{\lambda_i})$ follow from
part (c) in \cite[Theorem~3.1]{MRR}. The sign characteristics
related to the blocks $J_1(0)$, $J_2(0)$ and $J_3(0)$ follow from
part (d), (f), and (e), respectively, in \cite[Theorem~3.1]{MRR}.
\end{proof}

\section{$H$-nonnegative square roots}
In this section we study $H$-nonnegative square roots of $H$-nonnegative matrices.
\begin{theorem}\label{ThmHnonConditions}
Let $B$ be an $H$-nonnegative matrix in $\mathbb C ^{n\times n}$. Then $B$ has an $H$-nonnegative square root if and only if the following properties hold:
\begin{enumerate}
\item[\rm (i)] $B$ has no negative eigenvalues, that is, $\sigma(B)\subset [0,\infty)$;
\item[\rm (ii)] $B$ has no Jordan blocks of size two at the eigenvalue zero, that is, $t=0$ in the canonical form of $(B,H)$ as given in Theorem~\ref{ThmHnonneg}.
\end{enumerate}
\end{theorem}

\begin{proof}
Suppose that (i) and (ii) hold. Then by Theorem~\ref{ThmHsaConditions} (also, Theorem~\ref{Thm3.1inMRR}) $B$ has an $H$-selfadjoint square root $A$. To prove that $A$ is $H$-nonnegative we need to show that $A$ has a real spectrum and that the canonical form of $(A,H)$ agrees with that given in Theorem~\ref{ThmHnonneg}. That $A$ has a real spectrum is trivial from $\sigma(B)\subset[0,\infty)$. From (ii) the only pairs that can occur in any Segre pairing of $B$ are  $(1,0)$ and $(1,1)$. Then Theorem~\ref{ThmDescrSquareRoot} implies that the Jordan form of $A_0$ can only contain Jordan blocks of the form $J_1(0)$ and $J_2(0)$. From part (c) in \cite[Theorem~3.1]{MRR} the square roots of the Jordan blocks corresponding to the positive eigenvalues all have positive sign characteristic  and from part (f) each square root of the pairs $J_1(0)\oplus J_1(0)$ (which is associated with the pair $(1,1)$ and has Jordan form $J_2(0)$) also has positive sign characteristic. Therefore the canonical form of $(A,H)$ has the desired  form as given in (iii) of Theorem~\ref{ThmHnonneg}.

Conversely, suppose $B$ has an $H$-nonnegative square root $A$. By Theorem~\ref{ThmHsaConditions} $\sigma(B)\subset[0,\infty)$ and $s_+\geq t$. But by Theorem~\ref{ThmHnonneg} the Jordan form of $A$ contains no blocks $J_3(0)$ and $J_4(0)$, which by Theorem~\ref{ThmDescrSquareRoot} would originate from pairs $(2,1)$ and $(2,2)$ in a Segre pairing of $B$. Therefore $B$ can have no Jordan blocks of size two at the eigenvalue zero.
\end{proof}

\begin{remark}
The ``only if" direction in Theorem~\ref{ThmHnonConditions} may also be deduced from earlier results. Indeed, suppose $B$ is $H$-nonnegative, and has an $H$-nonnegative square root $A$. Then $B=A^2=A^{[*]}A$. Moreover, $A$ admits a trivial {semidefinite} $H$-polar decomposition: $A=I\cdot A$ in terms of Section 5 in \cite{BMRRR3}. Hence it follows from Theorem~5.3 in \cite{BMRRR3} that $B$ has all its eigenvalues in $[0,\infty)$ and is diagonalizable, which implies in particular that the Jordan blocks at zero are all of size one. In fact, from Theorem 5.3 in \cite{BMRRR3} one extra point may be concluded, namely that $\dim\textup{Ker} A\geq \max(s_+,s_-)$.
\end{remark}

By Theorem~\ref{ThmHnonConditions} there exists an $H$-nonnegative square root for an $H$-nonnegative matrix $B$ if and only if the canonical form of the pair $(B,H)$ is given by
\begin{align}\label{eq5HnonBH}
S^{-1}BS&=\bigoplus_{i=1}^qJ_1(\lambda_i)\oplus\bigoplus_{i=1}^sJ_1(0),\\
S^*HS&=\bigoplus_{i=1}^qQ_1\oplus\bigoplus_{i=1}^s\varepsilon_iQ_1,\label{eq5HnonH}
\end{align}
for some invertible matrix $S$, where $\lambda_i>0$ for all $i=1,\ldots,q$ and $\varepsilon_i=\pm 1$ for all $i=1,\ldots,s$. Note that the Segre characteristic of $B$ is equal to $(1,\ldots,1)$. This implies
the following theorem.

\begin{theorem}\label{ThmDescrHnon}
Let $B_0$ be a nilpotent $H$-nonnegative matrix in Jordan form satisfying the conditions in Theorem~\ref{ThmHnonConditions} and having a Segre pairing $\mathcal{S}$.
Any pair of $\mathcal{S}$ is of the form
\begin{equation*}
 (1,1) \mbox{ or } (1,0).
\end{equation*}
This corresponds to Jordan blocks of the form
$J_1(0)\oplus J_1(0)$ and $J_1(0)$, respectively, in the matrix $B_0$.
Then the $H$-nonnegative square roots associated with each possible pair in $\mathcal{S}$ are given by the first two items in Theorem~\ref{ThmDescrHsa}.
\end{theorem}

We now also give the canonical form of the $H$-nonnegative square roots. Let $\mathcal{S}$ be any Segre pairing of an $H$-nonnegative matrix $B$. Once again, let $\ell_1$ denote the number of pairs $(1,0)$ in $\mathcal{S}$ and $\ell_2$ the number of pairs $(1,1)$ in $\mathcal{S}$.

\begin{theorem}\label{ThmJordanFHnonSq}
Let $B$ be an $H$-nonnegative matrix satisfying the conditions in Theorem~\ref{ThmHnonConditions}. Assume that $(B,H)$ is in canonical form, i.e.  $S=I$ in \eqref{eq5HnonBH} and \eqref{eq5HnonH}. Then the canonical form of the pair $(A,H)$, where $A$ is an $H$-nonnegative square root of $B$, is given by:
\begin{equation*}
P^{-1}AP=\bigoplus_{i=1}^q J_1(\sqrt{\lambda_i})\oplus\bigoplus_{i=1}^{\ell_1}J_1(0)\oplus\bigoplus_{i=1}^{\ell_2}J_2(0),
\end{equation*}
\begin{equation*}
P^*HP=\bigoplus_{i=1}^qQ_1\oplus\bigoplus_{i=1}^{\ell_1}\zeta_{i} Q_1\oplus\bigoplus_{i=1}^{\ell_2}Q_2,
\end{equation*}
for some invertible matrix $P$, where $\zeta_i$ is equal to the sign characteristic of the corresponding block $J_1(0)$ in $B$.
\end{theorem}

\begin{proof}
As  $H$-nonnegative square roots of $B$ are also $H$-selfadjoint, the form of $P^{-1}AP$ follows from
Theorem~\ref{EternalLoveAngelicaZapata} together with Theorem~\ref{ThmHnonConditions} and the fact
that the $H$-nonnegative matrix $A$ cannot contain Jordan blocks related with negative eigenvalues
of positive sign characteristic, see Theorem~\ref{ThmHnonneg}.
Moreover, Theorem~\ref{EternalLoveAngelicaZapata}  shows
also the sign characteristic related to the
blocks $J_1(\sqrt{\lambda_i})$ and $J_1(0)$.  The sign characteristic related to $J_2(0)$
was shown in the proof of Theorem~\ref{ThmHnonConditions}.
\end{proof}


\section{Stability of $H$-nonnegative square roots}
We give a definition for unconditional stability. Later, we relax this definition and consider conditional stability.

\begin{definition}
An $H$-nonnegative square root $A$ of an $H$-nonnegative matrix $B$ is called \emph{unconditionally stable} if for all $\varepsilon>0$ there exists a $\delta>0$ such that for all ${H}$-nonnegative matrices $\tilde{B}$ satisfying $\norm{B-\tilde{B}}<\delta$, $\tilde{B}$ has an ${H}$-nonnegative square root $\tilde{A}$ for which $\norm{A-\tilde{A}}<\varepsilon$.
\end{definition}

Let $B$ be an $H$-nonnegative matrix which satisfies the conditions in Theorem~\ref{ThmHnonConditions}. Assume that $(B,H)$ is in canonical form. Let $\mathcal{S}$ be a Segre pairing of the matrix $B$. Let $\ell_1$ be the number of $(1,0)$ pairs in $\mathcal{S}$ and $\ell_2$ the number of $(1,1)$ pairs. Then we can write
\begin{equation}\label{CuandoVeoATuMama}
B=B_1\oplus B_2\oplus B_3=\bigoplus_{i=1}^{\ell_1}J_1(0)\oplus\bigoplus_{i=1}^{\ell_2}\left(J_1(0)\oplus J_1(0)\right)\oplus\bigoplus_{i=1}^q J_1(\lambda_i),
\end{equation}
\begin{equation*}
H=H_1\oplus H_2\oplus H_3=\bigoplus_{i=1}^{\ell_1}\delta_iQ_1\oplus\bigoplus_{i=1}^{\ell_2}(Q_1\oplus -Q_1)\oplus\bigoplus_{i=1}^q Q_1,
\end{equation*}
where $\lambda_i>0$, $\delta_i=\pm 1$. Note that by Theorem~\ref{ThmDescrHnon} the sign characteristics of the different blocks in $B$ are as in Theorem~\ref{ThmDescrHsa}, but this can also be seen from  \cite[Theorem~3.1]{MRR}.

We follow the proof of  \cite[Theorem~3.3]{MRR} for $H$-selfadjoint square roots with a few adjustments to obtain the following result.

\begin{theorem}\label{ThmStab1Uncon}
Let $B$ be an $H$-nonnegative matrix. Then there exists an unconditionally stable $H$-nonnegative square root of $B$ if and only if in \eqref{CuandoVeoATuMama}
\begin{equation}\label{Maluma}
\ell_2=0 \quad \mbox{and} \quad \delta_i=1.
\end{equation}
\end{theorem}

\begin{proof}
The blocks $(B_1,H_1)$, $(B_2,H_2)$ and $(B_3,H_3)$ can be considered separately.
Consider the matrices $B_1=J_1(0)=[0]$ (associated with the pair $(1,0)$) and $H_1=-Q_1=[-1]$. Then $B_1$ is $H_1$-nonnegative. Let $B_1(a)=[-a]$ be a continuous family of  matrices with $a\in[0,1]$. Then $B_1(a)$ is $H_1$-nonnegative and $B_1(0)=B_1$. But $B_1(a)$ does not satisfy the conditions in Theorem~\ref{ThmHnonConditions} and therefore no $H_1$-nonnegative square root of $B_1(a)$ exists for $a>0$. This means that no $H_1$-nonnegative square root of $B_1$ is stable.

Now consider the matrices $B_2=J_1(0)\oplus J_1(0)=\left[\begin{smallmatrix}
0&0\\0&0
\end{smallmatrix}\right]$ (associated with the pair $(1,1)$) and
$$
H_2=Q_1\oplus -Q_1=\begin{bmatrix}
1&0\\0&-1
\end{bmatrix}.
$$
Let
$$
B_2(a)=\begin{bmatrix}
a&0\\0&-a
\end{bmatrix},\  a\in[0,1],
$$
be a continuous family of matrices. Then $B_2(a)$ is $H_2$-nonnegative and $B_2(0)=B_2$. Once again, though, $B_2(a)$ does not satisfy the conditions in Theorem~\ref{ThmHnonConditions}. Therefore $B_2(a)$ does not have an $H_2$-nonnegative square root for $a>0$ and this means that no $H_2$-nonnegative square root of $B_2$ is stable.
This implies~\eqref{Maluma}.

Conversely, let $B_3$ be an $H_3$-nonnegative matrix with only positive eigenvalues, thus $H_3$ is the identity
and $B_3$ is positive definite (see~\eqref{CuandoVeoATuMama}). Let $\Gamma$ be a closed contour which encloses all eigenvalues of $B_3$ and which lies in the open right half plane. Let $\delta>0$ be a constant such that all of the eigenvalues of any matrix $\tilde{B}_3$ for which $\norm{B_3-\tilde{B}_3}<\delta$ holds, are enclosed by $\Gamma$. Given an $H_3$-nonnegative matrix $\tilde{B}_3$ which satisfies $\norm{B_3-\tilde{B}_3}<\delta$.
Then $\tilde{B}_3$ is also positive definite.  Thus, $B_3$ and $\tilde{B}_3$ each has a unique positive definite square root $A_3$ and $\tilde{A}_3$, respectively.
Let $\sqrt{\cdot}$ has its branch cut along the negative real line.  Then
\begin{equation*}
A_3=\frac{1}{2\pi i}\int_\Gamma \sqrt{z}(zI-B_3)^{-1}dz\quad\textup{and}\quad \tilde{A}_3=\frac{1}{2\pi i}\int_\Gamma \sqrt{z}(zI-\tilde{B}_3)^{-1}dz.
\end{equation*}
Consider
\begin{eqnarray*}
\norm{A_3-\tilde{A}_3}&=&\abs{\frac{1}{2\pi i}}\norml{\int_\Gamma \sqrt{z}[(zI-B_3)^{-1}-(zI-\tilde{B}_3)^{-1}]dz}\\
&\leq& \abs{\frac{1}{2\pi i}}\int_\Gamma \norml{\sqrt{z}(zI-B_3)^{-1}(B_3-\tilde{B}_3)(zI-\tilde{B}_3)^{-1}}dz\\
&\leq&\norm{B_3-\tilde{B}_3}\abs{\frac{1}{2\pi i}}\norml{\int_\Gamma \sqrt{z}(zI-B_3)^{-1}(zI-\tilde{B}_3)^{-1}dz}\\
&=&M\norm{B_3-\tilde{B}_3},
\end{eqnarray*}
and note that $M$ is positive. Therefore we have that the $H_3$-nonnegative square root $A_3$ is stable.

For the case where $H_1=Q_1=[1]$, the matrix $B_1$ is also $H_1$-nonnegative and the only square root of $B_1$ is $A_1=[0]$. Consider a sequence converging to zero entailing only nonnegative numbers. Now form a sequence by taking the positive square root of each number in that sequence. Then this sequence also converges to zero. This is enough to prove that for every $B_k$ converging to $B_1$ there exists an $H_1$-nonnegative square root $A_k$ that converges to $A_1$, i.e.\ $A_1$ is stable.
\end{proof}

We now define conditional stability to ensure that the matrix close to $B$ admits an $H$-nonnegative square root.
\begin{definition}
An $H$-nonnegative square root $A$ of an $H$-nonnegative matrix $B$ is called \emph{conditionally stable} if for all $\varepsilon>0$ there exists a $\delta>0$ such that for all $H$-nonnegative matrices $\tilde{B}$  having at least one $H$-nonnegative square root, and  satisfying $\norm{B-\tilde{B}}<\delta$, $\tilde{B}$ has an $H$-nonnegative square root $\tilde{A}$ for which $\norm{A-\tilde{A}}<\varepsilon$.
\end{definition}

In the following theorem we find conditions for conditional stability of $H$-nonnegative square roots of $H$-nonnegative matrices. Again, we closely follow the proof for $H$-selfadjoint square roots in \cite{MRR}.

\begin{theorem}\label{ThmStab2Con}
Let $B$ be an $H$-nonnegative matrix. Then there exists a conditionally stable $H$-nonnegative square root $A$ of $B$ if and only if in \eqref{CuandoVeoATuMama}
$$
\ell_2=0.
$$
\end{theorem}

\begin{proof}
Let $B_1=J_1(0)=[0]$ and $H_1=\delta Q_1=[\delta]$, with $\delta=\pm1$. The only square root of $B_1$ is $A_1=[0]$. If $\delta=1$, by  Theorem~\ref{ThmStab1Uncon},
 $A_1$ is unconditionally stable and, hence also conditionally stable.
 If $\delta=-1$, the only $H_1$-nonnegative matrix having an $H_1$-nonnegative square root, is the zero  matrix. Therefore consider the constant sequence $0,0,\ldots$. Then it follows that the $H_1$-nonnegative square root $A_1$ is conditionally stable.

Next, let $B_2=J_1(0)\oplus J_1(0)$ and $H_2=Q_1\oplus -Q_1$. We know from Theorem~\ref{ThmDescrHnon} that any $H_2$-nonnegative square root of $B_2$ is of the form
$$
A_2=\begin{bmatrix}
\alpha&-\bar{\beta}\\\beta&-\alpha
\end{bmatrix}\mbox{ for any } \alpha\in\mathbb{R},\ \beta\in\mathbb{C}\mbox{  with }\abs{\beta}=\pm\alpha.
$$
Note that, according to the Agreement in Section~\ref{SectionSqR}, we always consider $\alpha \neq 0$.
Consider a continuous family of $H_2$-nonnegative matrices
$$
B(a)=\begin{bmatrix}
a&0\\0&0
\end{bmatrix},\ a\in[0,1].
$$
 The only $H_2$-nonnegative square root of $B(a)$ is $\left[\begin{smallmatrix}
\sqrt{a}&0\\0&0
\end{smallmatrix}\right]$ but this matrix is equal to the zero matrix (and not $A_2$) if $a=0$. Therefore, no $H_2$-nonnegative square root of $B_2$ is conditionally stable.

Lastly, let $B_3$ be an $H_3$-nonnegative matrix with only positive eigenvalues.
By  Theorem~\ref{ThmStab1Uncon} there exists a
unconditionally stable $H_3$-nonnegative square root $A_3$
which is then  also conditionally stable.
\end{proof}

\paragraph{Acknowledgements:}
This work is based on research supported in part by the DSI-NRF Centre of Excellence in Mathematical and Statistical Sciences (CoE-MaSS). Opinions expressed and conclusions arrived at are those of the authors and are not necessarily to be attributed to the CoE-MaSS.





\end{document}